\documentclass[12]{amsart}

\usepackage{amsthm}

\newtheorem{theo}{Theorem}
\newcounter{tmp}

\newtheorem{theorem}{Theorem}[section]
\newtheorem {lemma}[theorem]{Lemma}
\newtheorem{corollary}[theorem]{Corollary}
\newtheorem{prop}[theorem]{Proposition}

\theoremstyle{definition}
\newtheorem{definition}[theorem]{Definition}
 
\theoremstyle{remark}
\newtheorem  {remark}[theorem]{Remark}

\numberwithin{equation}{section}

\newcommand{\C}{\mathbb C}
\newcommand{\RS}{\widehat{\C}}
\newcommand{\D}{\Delta}

\newcommand{\T}{Teich\-m\"ul\-ler}
\newcommand{\hm}{hol\-o\-mor\-phic}

\begin{document}

\title[Monodromy, Lifting, And Holomorphic motions]{MONODROMY, LIFTINGS OF HOLOMORPHIC MAPS, AND
EXTENSIONS OF HOLOMORPHIC MOTIONS}

\author{Yunping Jiang}
\address[Jiang]{Department of Mathematics, Queens College of the City
University of New York, USA\\
and Department of Mathematics, The Graduate Center, CUNY, USA}
\email{yunping.jiang@qc.cuny.edu}

\author{Sudeb Mitra}
\address[Mitra]{Department of Mathematics, Queens College of the City University of New York, USA\\
and Department of Mathematics, The Graduate Center, CUNY, USA}
\email{sudeb.mitra@qc.cuny.edu}


\thanks{This material is based upon work supported by the National Science Foundation. The first author is also partially supported by a
collaboration grant from the Simons Foundation (grant number 523341) and PSC-CUNY awards and a grant from
NSFC (grant number 11571122).}

\begin{abstract}
We study monodromy of holomorphic motions and show the equivalence
of triviality of monodromy of holomorphic motions and extensions of
holomorphic motions to continuous motions of the Riemann sphere. We also
study liftings of holomorphic maps into certain Teichm\"uller spaces. We use
this ``lifting property'' to prove that, under the condition of trivial monodromy,
any holomorphic motion of a closed set in the Riemann sphere, over a hyperbolic
Riemann surface, can be extended to a holomorphic motion of the sphere,
over the same parameter space. We conclude that this extension can be done
in a conformally natural way.
\end{abstract}

\maketitle

\section*{Introduction}
Throughout this paper, we shall use the following notations: $\C$ for the complex plane, $\RS = \C \cup \{\infty\}$ for the Riemann sphere, and $\Delta$ for the open unit disk $\{z \in \C: \vert z\vert < 1\}$. 

The subject of holomorphic motions was introduced in the study of the dynamics
of rational maps; see~\cite{MSS}. Since its inception, an important topic has been the
question of extending holomorphic motions. The papers~\cite{BR} and~\cite{ST} contained
partial results. Subsequently, Slodkowski showed that any holomorphic motion of a
set in $\RS$, over $\D$ as the parameter space, can be extended to a holomorphic motion
of $\RS$ over $\D$; see~\cite{Sl} and~\cite{Ch}. The paper~\cite{EM} used a group-equivariant version of Slodkowski's
theorem to prove results in Teichm\"uller theory. 

The main purpose of our paper is to study necessary and sufficient topological
conditions for extending holomorphic motions. We study monodromy of a holomorphic
motion $\phi$ of a finite set $E$ in $\RS$, defined over a connected complex Banach
manifold $V$ . We show the equivalence of the triviality of this monodromy and the
extendability of $\phi$ to a continuous motion of $\RS$ over $V$ . We then show that, if $V$ is
a hyperbolic Riemann surface, and $E$ is any set in $\RS$, then, under the condition of
trivial monodromy, any holomorphic motion of $E$ extends to a holomorphic motion
of $\RS$ over $V$ . The main technique is to study the liftings of holomorphic maps from
a hyperbolic Riemann surface into Teichm\"uller spaces of $\RS$ with punctures.

Our paper is organized as follows. In Section 1, we give all precise definitions and
discuss the useful facts that will be necessary in the proofs of the main theorems
of this paper. In Section 2, we give precise statements of the main theorems. In Sections
3, 4, 5, and 6, we prove the main theorems.

\section{Definitions and some facts}

\medskip
\begin{definition}~\label{hm} 
Let $V$ be a connected complex manifold with a basepoint $x_0$ and let $E$ be any subset of $\RS$. A {\it holomorphic motion} of $E$ over $V$ is a map $\phi: V \times E \to \RS$ that has the following three properties:
\begin{enumerate}
\item $\phi(x_0,z) = z$ for all $z$ in $E$,
\item the map $\phi(x,\cdot): E \to \RS$ is injective for each $x$ in $V$, and
\item the map $\phi(\cdot,z): V \to \RS$ is holomorphic for each $z$ in $E$.
\end{enumerate}
\end{definition}

We say that $V$ is a {\it parameter space} of the holomorphic motion $\phi$. We will always assume that $\phi$ is a {\it normalized} holomorphic motion; i.e. 0, 1, and $\infty$ belong to $E$ and are fixed points of the map $\phi(x, \cdot)$ for every $x$ in $V$. It is sometimes useful to write $\phi(x,z)$ as $\phi_x(z)$.

If $E$ is a proper subset of $\widehat E$ and $\phi: V \times E \to \RS$, $\widehat\phi: V \times \widehat E \to \RS$ are two holomorphic motions, we say that $\widehat\phi$ {\it extends}  $\phi$ if $\widehat\phi(x,z) = \phi(x, z)$ for all $(x ,z)$ in $V \times E$.
 
\medskip
\begin{remark}~\label{pb} 
Let $V$ and $W$ be connected complex manifolds with
basepoints, and $f$ be a basepoint preserving {\hm} map of
$W$ into $V$. If $\phi$ is a {\hm} motion of $E$ over $V$
its {\it pullback} by $f$ is the {\hm} motion
\begin{equation}\label{eq:pb}
f^*(\phi)(x,z) =\phi(f(x),z), \qquad \forall (x,z) \in W \times E,
\end{equation}
of $E$ over $W$.
\end{remark}

\medskip
\begin{definition}~\label{ge}
Let $V$ be a connected complex manifold with a basepoint. Let $G$ be a group of M\"obius transformations, let $E \subset \RS$ be $G$-invariant, which means, $g(E) = E$ for each $g$ in $G$. A holomorphic motion $\phi: V \times E \to \RS$ is {\bf $G$-equivariant} if for any $x \in V, g \in G$ there is a M\"{o}bius transformation, denoted by $\theta_{x}(g)$, such that 
$$ \phi(x, g(z)) = (\theta_x(g))(\phi(x,z))$$
for all $z$ in $E$.
\end{definition}

It is well-known that if $\phi: V \times E \to \RS$ is a {\hm} motion, where $V$ is a connected complex manifold with a basepoint $x_0$, then $\phi$ extends to a {\hm} motion of the closure $\overline{E}$, over $V$; see~\cite{MSS} and ~\cite{BJM}. Hence, throughout this paper, we will assume that $E$ is a closed set in $\RS$ (that contains the points $0$, $1$, and $\infty$). 

\medskip
Recall that a homeomorphism of $\RS$ is called {\it normalized} if it fixes the
points 0, 1, and $\infty$. The blanket assumption that $E$ is a closed set in $\RS$ containing the points $0$, $1$, and $\infty$ holds.

\medskip
\begin{definition}~\label{ts}
Two normalized quasiconformal self-mappings $f$ and $g$ of $\RS$
are said to be $E$-equivalent if and only if $f^{-1} \circ g$ is
isotopic to the identity rel $E$. The {\it Teichm\"uller space}
$T(E)$ is the set of all $E$-equivalence classes of normalized
quasiconformal self-mappings of $\RS$.
\end{definition}

Let $M(\C)$ be the open unit ball of the complex Banach space
$L^{\infty}(\C)$. Each $\mu$ in $M(\C)$ is the Beltrami coefficient
of a unique normalized quasiconformal homeomorphism $w^{\mu}$ of
$\RS$ onto itself. The basepoint of $M(\C)$ is the zero function.

We define the quotient map
$$
P_E: M(\C) \to T(E)
$$
by setting $P_E(\mu)$ equal to the
$E$-equivalence class of $w^{\mu}$, written as $[w^{\mu}]_E$.
Clearly, $P_E$ maps the basepoint of $M(\C)$ to the basepoint of
$T(E)$.

In his doctoral dissertation~\cite{L}, G.~Lieb proved that $T(E)$
is a complex Banach manifold such that the projection map $P_E: M(\C)
\to T(E)$ is a holomorphic split submersion; see \cite{EM} for the details. 

\medskip
\begin{remark} 
Let $E$ be a finite set. Its complement $\Omega= \RS \setminus E$ is the Riemann
sphere with punctures at the points of $E$. Then, $T(E)$ is biholomorphic to the classical Teichm\"uller space $Teich(\Omega)$; see Example 3.1 in \cite{M1} for the proof. This canonical identification will be very important in our paper. 
\end{remark}

\medskip
\begin{prop}~\label{edes}  
There is a continuous basepoint preserving map $s$ from $T(E)$ to $M(\C)$ such that $P_E \circ s$ is the identity map on $T(E)$.  \end{prop}

See~\cite{EM},~\cite{JM} for all details. 

\medskip
\begin{definition}~\label{des}
The map $s$ from $T(E)$ to $M(\C)$ is called the {\it Douady-Earle section} of $P_E$ for the {\T} space $T(E)$. 
\end{definition}

\medskip
\begin{remark}~\label{sec}
When $E$ is finite, $T(E)$ is canonically identified with the classical {\T} space $Teich(\RS \setminus E)$, and hence $s$ is the continuous section studied in Lemma 5 in \cite{DE}.
\end{remark}

\medskip
\begin{definition}~\label{uhm}  
The {\it universal holomorphic motion} $\Psi_E: T(E) \times E \to \RS$ is defined as follows:
$$
\Psi_E(P_E(\mu),z) = w^{\mu}(z) \hbox{ for $\mu \in M(\C)$ and $z \in E$}.
$$
\end{definition}

It is clear from the definition of $P_E$ that the map $\Psi_E$ is
well-defined. It is a holomorphic motion because $P_E$ is a
holomorphic split submersion and $\mu \mapsto w^{\mu}(z)$ is a
holomorphic map from $M(\C)$ to $\RS$ for every fixed $z$ in
$\RS$, by Theorem 11 in \cite{AB}. 

This holomorphic motion is ``universal" in the following sense:

\medskip
\begin{theorem}~\label{univ}
Let $\phi: V \times E \to \RS$ be a holomorphic
motion. If $V$ is a simply connected complex Banach manifold with a
basepoint, there is a unique basepoint preserving holomorphic map $f:
V \to T(E)$ such that $f^*(\Psi_E) = \phi$. \end{theorem}

For a proof see Section 14 in \cite{M1}.

\medskip
\begin{remark}~\label{qcrep}
Let $\phi: V \times E \to \RS$ be a {\hm} motion where $V$ is a connected complex Banach manifold with a basepoint $x_0$. 
Suppose there exists a basepoint preserving {\hm} map $f: V \to T(E)$ such that $f^*(\Psi_E) = \phi$. 
Let $\widetilde f: V \to M(\C)$ where $\widetilde f = s \circ f$. 
By Proposition~\ref{edes}, $\widetilde f$ is a basepoint preserving continuous map. Then, for all $(x,z) \in V \times E$, we have 
$$
\phi(x,z) = \Psi_E(f(x), z) = \Psi_E(P_E(s(f(x))),z) = w^{s(f(x))}(z) = w^{\widetilde f(x)}(z).
$$ 
\end{remark}

\medskip
\begin{definition}~\label{cm}
Let $W$ be a path-connected Hausdorff space with a basepoint $x_0$. A (normalized) continuous motion of $\RS$ over $W$ is a continuous map $\phi: W \times \RS \to \RS$ such that:
\begin{itemize}
\item[(a)] $\phi(x_0,z) = z$ for all $z \in \RS$, and 
\item[(b)] for each $x$ in $W$, the map $\phi(x,\cdot): = \phi_x(\cdot)$ is a homeomorphism of $\RS$ onto itself that fixes the points $0$, $1$, and $\infty$. 
\end{itemize}
\end{definition}

In \cite{M2} it was shown that:

\medskip
\begin{theorem}~\label{ceq}
Let $\phi: V \times E \to \RS$ be a {\hm} motion where $V$ is a connected complex Banach manifold with a basepoint $x_0$. Then the following are equivalent:
\begin{itemize}
\item[(i)] There is a continuous motion $\widetilde\phi: V \times \RS \to \RS$ that extends $\phi$.
\item[(ii)] There exists a basepoint preserving {\hm} map $F: V \to T(E)$ such that $F^*(\Psi_E) = \phi$. 
\end{itemize}
\end{theorem}

The following corollary will be useful in this paper; see~\cite{M2}. 

\medskip
\begin{corollary}~\label{conbel}
If the {\hm} motion $\phi$ can be extended to a continuous motion $\widetilde\phi$, then $\widetilde\phi$ can be chosen so that:
\begin{itemize}
\item[(i)] the map $\widetilde\phi_x: \RS \to \RS$ is quasiconformal for each $x$ in $V$, and 
\item[(ii)] its Beltrami coefficient $\mu_x$ is a continuous function of $x$. 
\end{itemize}
\end{corollary}

Let $w$ be a normalized quasiconformal self-mapping of $\RS$, and let $\widetilde E = w(E)$. By definition, the {\it allowable map} $g$ from $T(\widetilde E)$ to $T(E)$ maps the $\widetilde E$-equivalence class of $f$ to the $E$-equivalence class of $f\circ w$ 
for every normalized quasiconformal self-mapping $f$ of $\RS$.  

\medskip
\begin{lemma} 
The allowable map $g: T(\widetilde E) \to T(E)$ is biholomorphic. If $\mu$ is the Beltrami coefficient of $w$, then $g$ maps the basepoint of $T(\widetilde E)$ to the point $P_E(\mu)$ in $T(E)$. 
\end{lemma}

See \S7.12 in \cite{EM} or \S6.4 in \cite{M1} for a complete proof. 
The following lemma will be useful in our paper.

\medskip
\begin{lemma}~\label{same} 
Let $B$ be a path-connected topological space. Let $f$ and $g$ be two continuous maps from $B$ to $T(E)$, satisfying:
\begin{itemize}
\item[(i)] $\Psi_E(f(t),z) = \Psi_E(g(t),z)$ for all $z$ in $E$, and
\item[(ii)] $f(t_0) = g(t_0)$ for some $t_0$.
\end{itemize}
Then, $f(t) = g(t)$ for all $t$ in $B$. 
\end{lemma}

See \S12 in~\cite{M1} for the proof. 

If $f(t) = [w^{\mu}]_E$ and $g(t) = [w^{\nu}]_E$, Condition (i) of the lemma means that $w^{\mu}(z) = w^{\nu}(z)$ for all $z$ in $E$. 

If $E$ is a subset of the closed set $\widehat E$ and $\mu$ is in $M(\C)$, then the $\widehat E$-equivalence class of $w^{\mu}$ is contained in the $E$-equivalence class of $w^{\mu}$. Therefore, there is a well-defined `forgetful map' 
\begin{equation}~\label{projee}
p_{\widehat E, E}: T(\widehat E)\mapsto T(E)
\end{equation}
such that $P_E = p_{\widehat{E},E} \circ P_{\widehat{E}}$. It is easy to see that this is a basepoint preserving holomorphic split submersion. 

%Condition (ii) of Lemma~\ref{same} means that $w^{\mu}$ and $w^{\nu}$ are isotopic rel $E$ at some point $t_{0}\in B$. The conclusion of Lemma~\ref{same} says that $w^{\mu}$ and $w^{\nu}$ are isotopic rel $E$ at for all $t\in B$.
%If $E$ is a subset of $\widehat {E}$ (as usual, we assume that 0, 1, and $\infty$ belong to both $E$ and $\widehat {E}$). Recall $p_{\widehat E, E}$ from (\ref{projee}). 
%Remember that $\Psi_E$ is the universal holomorphic motion of $E$ and $\Psi_{\widehat E}$ is the universal motion of $\widehat E$. 
The following is a consequence of Lemma~\ref{same}. Here, $\Psi_E$ is the universal holomorphic motion of $E$ and $\Psi_{\widehat E}$ is the universal motion of $\widehat E$.

\medskip 
\begin{prop}~\label{le}
Let $V$ be a connected complex Banach manifold with basepoint, and let $f$ and $g$ be basepoint preserving holomorphic maps from $T(E)$ and $T(\widehat E)$ respectively. Then $ p_{\widehat E,E} \circ g = f$ if and only if $g^*(\Psi_{\widehat E})$ extends $f^*(\Psi_E)$. 
\end{prop}

See \S13 in \cite{M1} for the proof. We say that the {\hm} map $g$ {\it lifts} the {\hm} map $f$. 

We now discuss the concept of {\it monodromy} of a {\hm} motion. We closely follow the discussion in \S2 in \cite{BJMS}. 
Let $\phi: V \times E \to \RS$ be a {\hm} motion, where $V$ is a connected complex Banach manifold with a basepoint $x_0$. Let $\pi: \widetilde V \to V$ be a {\hm} universal covering, with the group of deck transformations $\Gamma$. We choose a point $\widetilde x_0$ in $\widetilde V$ such that $\pi(\widetilde x_0) = x_0$. Let $\pi_1(V, x_0)$ denote the fundamental group of $V$ with basepoint $x_0$.

Let $\Phi = \pi^{*}(\phi)$. Then, $\Phi: \widetilde V \times E \to \RS$ is a {\hm} motion of $E$ over $\widetilde V$ with $\widetilde x_0$ as the basepoint. By Remark~\ref{qcrep}, there exists a basepoint preserving continuous map $\widetilde f:\widetilde{V} \to M(\C)$ such that
$$
\Phi(x,z) = w^{\widetilde{f}(x)}(z)
$$
for each $x \in \widetilde V$ and each $z \in E$. 

For each $z \in E$ and for each $\gamma \in \Gamma$, we have 
$$w^{\widetilde{f} \circ \gamma(\widetilde x_0)}(z) = \Phi(\gamma(\widetilde x_0), z) = \phi(\pi \circ \gamma(\widetilde x_0), z) = \phi(x_0, z) = z.$$
Therefore, $w^{\widetilde{f} \circ \gamma(\widetilde x_0)}$ keeps every point of $E$ fixed. 

\medskip
\begin{lemma}~\label{indep} 
The homotopy class of $w^{\widetilde{f} \circ \gamma(\widetilde x_0)}$ relative to $E$ does not depend on the choice of the continuous map $\widetilde{f}$.
\end{lemma}

See Lemma 2.12 in~\cite{BJMS}.

\medskip

We now assume that $E$ is a finite set containing $n$ points where $n \geq 4$; as usual, $0$, $1$, and $\infty$ are in $E$. Let $\phi: V \times E \to \RS$ be a {\hm} motion. The map $w^{\widetilde{f} \circ \gamma(\widetilde x_0)}$ is quasiconformal selfmap of the hyperbolic Riemann surface $X_E:= \RS \setminus E$.  Therefore, it represents a mapping class of $X_E$, and by Lemma~\ref{indep}, we have a homomorphism $\rho_{\phi}: \pi_{1}(V, x_0) \to$ Mod$(0,n)$ given by
$$
\rho_{\phi}(c) = [w^{\widetilde{f} \circ \gamma(\widetilde x_0)}]
$$
where Mod$(0,n)$ is the mapping class group of the $n$-times punctured sphere, $\gamma \in \Gamma$ is the element 
corresponding to $c \in \pi_1(V, x_0)$, and $[w]$ denotes the mapping class group of $X_E$ for $w$. 

\medskip
\begin{definition}~\label{mon}
We call the homomorphism $\rho_{\phi}$ the {\it monodromy} of $\phi$
the holomorphic motion $\phi$ of the finite set E. The monodromy is called {\it trivial} if it maps every element of $\pi_1(V, x_0)$ to the identity of Mod$(0, n)$. 
\end{definition}

\section{Statements of the main theorems}

In this Section, we give the precise statements of the main theorems of our paper.
 
\begingroup
\setcounter{tmp}{\value{theo}}% store current value of theorem counter
\setcounter{theo}{0} %assign desired value to theorem counter
\renewcommand\thetheo{\Alph{theo}}% locally redefine the representation of the theorem counter
  
\medskip
\begin{theo}~\label{thmA}
Let $\phi: V \times E \to \RS$ be a {\hm} motion of a finite set $E$, containing the points $0$, $1$, and $\infty$, where $V$ is a connected complex Banach manifold with  basepoint $x_0$. The following are equivalent:
\begin{itemize}
\item[(i)] There exists a continuous motion $\widetilde\phi: V \times \RS \to \RS$, such that $\widetilde\phi$ extends $\phi$. 
\item[(ii)] The monodromy $\rho_{\phi}$ is trivial. 
\end{itemize}
\end{theo}

In the next theorem, let $\widehat E = E \cup \{\zeta\}$, where, $E$ is a finite set containing $0$, $1$, and $\infty$, and $\zeta \in \C \setminus E$. Let $V$ be a connected complex Banach manifold with basepoint $x_0$. 

\medskip
\begin{theo}~\label{thmB}
Suppose every {\hm} map from $V $ into $T(E)$ lifts to a {\hm} map from $V$ into $T(\widehat E)$. Then, if $\phi: V \times E \to \RS$ is a {\hm} motion that has trivial monodromy, there exists a {\hm} motion $\widehat\phi : V \times \widehat{E} \to \RS$ such that $\widehat\phi$ extends $\phi$ and also has trivial monodromy.
\end{theo}

In the next three theorems, $X$ is a hyperbolic Riemann surface with a basepoint $x_0$, and $E$ is a closed set in $\RS$ 
containing the points $0, 1$, and $\infty$.

Let  $E_{n}=\{0, 1, \infty, \xi_{1}, \cdots, \xi_{n}\}$ where $n\geq 1$, and $E_{n+1} =E_{n}\cup \{\xi_{n+1}\}$ where  $\xi_{n+1}\in\RS\setminus E_{n}$.  
 Let $p: T(E_{n+1})\to T(E_{n})$ denote the forgetful map in (\ref{projee}). Let $\phi_{n}: X\times E_{n}\to \RS$ 
be a {\hm} motion, that has trivial monodromy. We will see in the proof of Theorem~\ref{thmA} that there exists 
a unique basepoint preserving holomorphic map $f_{n}: X\to T(E_{n})$ such that $f_{n}^{*} (\Psi_{E_{n}})=\phi_{n}$.
The following  theorem is a key result in our paper.

\medskip
\begin{theo}~\label{thmC}
Let $\phi_{n}: X\times E_{n} \to \RS$ be a holomorphic motion. If the monodromy of $\phi_{n}$ is trivial, 
there exists a basepoint preserving holomorphic map $f_{n+1} : X \to T(E_{n+1})$ such that $p \circ f_{n+1} = f_n$. 
\end{theo}

The following corollary is an immediate consequence. Here $\Psi_{E_{n+1}}: T(E_{n+1})\times E_{n+1}\to \RS$ is the universal \hm motion of $E_{n+1}$. 

\medskip
\begin{corollary}~\label{exttrivial}
Let $\phi_{n+1}: =f^{*}_{n+1} (\Psi_{E_{n+1}})$. Then $\phi_{n}$ extends $\phi_{n+1}$  and has the trivial monodromy.
\end{corollary} 

%In the next theorem,  $E$ is a closed set in $\RS$ containing the points $0$, $1$, and $\infty$. 

\medskip
\begin{theo}~\label{thmD}
Let $\phi: X \times E \to \RS$ be a {\hm} motion such that $\phi$ restricted to $X \times E'$ has trivial monodromy, or extends to a continuous motion of $\RS$ (over $X$),
where $E'$ is any finite subset of $E$, containing the points $0, 1$, and $\infty$. Then, there
exists a {\hm} motion $\widehat{\phi}: X \times \RS \to \RS$ such that $\widehat{\phi}$ extends $\phi$. 
\end{theo}

\medskip
\begin{remark}
H. Shiga~\cite{S} has recently announced a completely different approach to a
part of this theorem. Our methods are totally independent and more direct. The
crucial point in our approach is the lifting property as given in Theorem~\ref{thmC}.
\end{remark}

\section{Proof of Theorem~\ref{thmA}}

Let $\pi: \widetilde V \to V$ be a {\hm} universal covering with the group $\Gamma$ of deck transformations, so that, $V = \widetilde{V}/\Gamma$, and $\pi(\widetilde x_0)  = x_0$.   

Suppose $\phi$ can be extended to a continuous motion $\widetilde\phi$ of $\RS$ over $V$. Then, by Corollary~\ref{conbel}, there exists a continuous map $f: V \to M(\C)$ such that $\widetilde\phi(x,z) = w^{f(x)}(z)$ for all $(x,z) \in V \times \RS$. Let $\widetilde{f} = f \circ \pi$. Then, for any $c \in \pi_1(V, x_0)$ with corresponding $\gamma \in \Gamma$, we have 
$$\rho_{\phi}(c) = [w^{\widetilde{f} \circ \gamma(\widetilde x_0)}] = [w^{f \circ \pi \circ \gamma(\widetilde x_0)}] = [w^{f(x_0)}] = [Id].$$ 
This shows that the monodromy $\rho_{\phi}$ is trivial.

Let $\phi:V  \times E \to \RS$ be a {\hm}  motion with trivial monodromy. Let $\phi_{\widetilde V}:= \pi^{*}(\phi)$ be the {\hm} motion of $E$ over $\widetilde V$. By Theorem~\ref{univ}, there exists a unique basepoint preserving {\hm} map $f_{\widetilde V}: \widetilde V \to T(E)$, such that $\phi_{\widetilde V} = f_{\widetilde V}^{*}(\Psi_{E})$. For any element $\gamma \in \Gamma$, we also have $f _{\widetilde V}\circ \gamma: \widetilde V \to T(E)$. Note that 
$$(f_{\widetilde V} \circ \gamma)^{*}(\Psi_{E})(x,z) = \Psi_{E}((f_{\widetilde V} \circ \gamma)x, z)
 = \phi_{\widetilde V}(\gamma(x), z) = \phi(\pi(\gamma(x)),z)$$ 
 $$= \phi(\pi(x), z) = \phi_{\widetilde V}(x,z) = (f_{\widetilde V})^{*}(\Psi_{E})(x,z).$$
By the triviality of the monodromy, we have $f_{\widetilde V} \circ \gamma(x_0) = f_{\widetilde V}(x_0)$ for all $\gamma \in \Gamma$. Lemma~\ref{same} implies that $f_{\widetilde V} \circ \gamma = f_{\widetilde V}$ for all $\gamma \in \Gamma$. Thus, $f_{\widetilde V}$ defines a unique basepoint preserving {\hm} map $f: V \to T(E)$ such that $\phi = f^*(\Psi_E)$. It then follows from Theorem~\ref{ceq} that there exists a continuous motion of $\RS$ over $V$ that extends $\phi$. This completes the proof. \qed

\section{Proof of Theorem~\ref{thmB}}

Let $\phi: V \times E \to \RS$ be a {\hm} motion such that it has trivial monodromy. By the proof of Theorem A, 
there exists a basepoint preserving {\hm} map $f: V \to T(E)$ such that $\phi = f^*(\Psi_{E})$. Let $p: T(\widehat E) \to T(E)$ 
denote the forgetful map defined in (\ref{projee}). By hypothesis, there exists a basepoint preserving {\hm} map $\widehat f: V \to T(\widehat E)$ such that $p \circ \widehat f = f$.  Let $\widehat\phi: = \widehat f^*(\Psi_{\widehat E})$. By Proposition~\ref{le}, $\widehat\phi$ extends $\phi$. Note that, for $x \in V$, and $z \in \widehat E$, we have $\widehat\phi(x,z) = w^{\alpha(x)}(z)$ where $\alpha = s_{\widehat E} \circ \widehat f$ and $s_{\widehat E}$ is the continuous section of the projection $P_{\widehat E}: M(\C) \to T(\widehat E)$; see Remark~\ref{sec}.

Let  $\pi: \widetilde V \to V$ be the {\hm} universal cover, and $\pi(\widetilde x_0) = x_0$. Then, $\pi^*(\widehat\phi): \widetilde V \times \widehat E \to \RS$ is a {\hm} motion. Since $\widetilde V$ is simply connected, there exists a basepoint preserving continuous map $\beta: \widetilde V \to M(\C)$ such that $\beta^*(\Psi_{\widehat E}) = \pi^*(\widehat\phi)$ (see Remark~\ref{qcrep}).
That implies $\pi^*(\widehat\phi)(x,z) = w^{\beta(x)}(z)$.  Recall that the monodromy $\rho: \pi_1(V) \to$ Mod$(0, n+1)$ is defined by
$$\rho(c) = [w^{\beta \circ \gamma(x_0)}]$$
 for any $c \in \pi_1(V, x_0)$ with the corresponding $\gamma \in \Gamma$.  Furthermore, it is independent of the choice of $\beta$. In particular, if we choose $\beta = \alpha \circ \pi$, we see that
 $$\rho(c) = [w^{\alpha \circ \pi \circ \gamma(x_0)}] = [w^{\alpha(x_0})] = [Id].$$
 Note that $\pi \circ \gamma(x_0) = x_0$. This implies that $\rho$ is trivial. \qed

\section{Proofs of Theorem~\ref{thmC} and Corollary~\ref{exttrivial}}

We recall the following result, due to S.~Nag, that will be fundamental in our paper; see \cite{N1}.

\medskip
\begin{theorem}~\label{nag} 
Given $n > 0$, choose a point $(\zeta_1, \cdot \cdot \cdot, \zeta_n)$ in the domain 
$$Y_n = \{(z_1, \cdot \cdot \cdot, z_n)\} \in \C^{n}: z_i \not= z_j \mbox { for } i \not= j \mbox { and } z_i \not= 0, 1 \mbox { for all } i = 1, \cdot \cdot \cdot, n\}$$
and let $E_n = \{0, 1, \infty, \zeta_1, \cdot \cdot \cdot, \zeta_n\}$. Then, the map $p_n: T(E_n) \to Y_n$ defined by
$$p_n([w^{\mu}]_{E_n}) = (w^{\mu}(\zeta_1), \cdot \cdot \cdot, w^{\mu}(\zeta_n)) \mbox { for all } \mu \in M(\C)$$
is a {\hm} universal covering. 
\end{theorem}

 Let $E_n = \{0, 1, \infty, \zeta_1, \cdot \cdot \cdot, \zeta_n\}$
and $E_{n+1} = E \cup \{\zeta_{n+1}\}$ where $\zeta_{n+1} \in \RS \setminus E_n$. We have also a {\hm} universal covering $p_{n+1}: T(E_{n+1}) \to Y_{n+1}$ where
$$Y_{n+1} = \{(z_1, \cdot \cdot \cdot, z_{n+1})\in \C^{n+1}: z_i \not= z_j \mbox { for } i \not= j \mbox { and } z_i \not= 0, 1 \mbox { for all } i = 1, \cdot \cdot \cdot, n+1\}$$

Let $p: T(E_{n+1}) \to T(E_{n})$ denote the forgetful map defined in (\ref{projee}).

We need some preliminaries. The reader is referred to Lemmas 3.1--3.5 in~\cite{JMW} for all details. 
Let $\mathcal{C}(\C)$ denote the complex Banach space of bounded, continuous functions $\phi$ on $\C$ with the norm
$$\Vert\phi\Vert = \sup_{z \in \C}\vert \phi(z)\vert.$$
In \S3 of \cite{JMW} we construct a continuous compact operator 
$$
\mathcal K: \mathcal C(\C) \to \mathcal C(\C).
$$
By Lemma 3.2 in \cite{JMW}, there exists a constant $C_3 > 0$ such that 
$$\Vert \mathcal K\Vert \leq C_3 \mbox { for all } f \in \mathcal C(\C).$$
For $\zeta_{n+1}$, let
$$\mathcal B = \{f \in \mathcal C(\C): \Vert f\Vert \leq \vert \zeta_{n+1}| + C_3\}.$$
It is a bounded convex subset in $\mathcal C(\C)$. The continuous compact operator $\zeta_{n+1} + \mathcal K$ maps $\mathcal B$ into itself. By Schauder fixed point theorem (see Theorem 2A on page 56 of~\cite{Z}; also page 557 of~\cite{JMW}), $\zeta_{n+1} + \mathcal K$ has a fixed point in $\mathcal B$. This says that we can find a $g_{n+1} \in \mathcal B$ such that 
$$
g_{n+1}(z) = \zeta_{n+1} + \mathcal K g_{n+1}(z) \mbox { for all } z \in \C.
$$
The reader is referred to \S3 of \cite{JMW} for all details; especially page 557 of that paper. By Lemma 3.5 of \cite{JMW}, the solution $g_{n+1}(z)$ is the unique fixed point of the operator $\zeta_{n+1} + \mathcal K$. 

Suppose $\phi_n: X \times E_n \to \RS$ is a  {\hm} motion, that has trivial monodromy. By Theorem~\ref{thmA}, it can be extended to a continuous motion $\widetilde{\phi}: X\times \RS\to \RS$. By Theorem~\ref{ceq}, that there exists a basepoint preserving holomorphic map $f_{n}: X\to T(E_{n})$ such that $f_{n}^{*} (\Psi_{E}) =\phi_{n}$. 

\begin{proof}[Proof of Theorem~\ref{thmC}]
Let 
$$
p_{n} : T ( E_{n})\to Y_{n}\quad \hbox{and}\quad  p_{n+1} : T ( E_{n +1}) \to Y_{n +1}
$$
be the two holomorphic coverings in Theorem~\ref{nag}.

Since $X$ is a hyperbolic Riemann surface, its universal cover space is the open unit disk $\Delta$. 
Let $\pi: \Delta\to X$ be the universal cover. Let $\Gamma$ be the group of deck transformations such that $X=\Delta/\Gamma$. 

Define the holomorphic map 
$$
f_{\Delta,n}= f_{n}\circ \pi : \Delta \to T(E_{n})
$$
and for any given $\gamma\in \Gamma$, consider the holomorphic map 
$$
f_{\Delta ,n} \circ \gamma:  \Delta \to T(E_{n}).
$$
This give us two holomorphic maps 
$$
F_{n} = p_{n} \circ f_{\Delta,n} : \Delta\to Y_{n}\quad \hbox{and}\quad F_{n}\circ \gamma = p_{n} \circ ( f_{\Delta ,n} \circ \gamma) : \Delta \to Y_{n}.
$$
Let us write 
$$
F_{n} = (h_{1} , \cdots , h_{n})\quad \hbox{and} \quad F_{n}\circ \gamma = (h_{1} \circ \gamma, \cdots, h_{n}\circ \gamma).
$$ 
Each $h_{i}$ (as well as $h_{i}\circ \gamma$) is holomorphic in $\Delta$. In \S3 of the paper [10], 
we constructed a map $g_{i}$ (as well as a map $g_{i}\circ \gamma$) by using $h_{i}$ (as well as $h_{i}\circ \gamma$) which is holomorphic outside $\Delta$ and continuous on $\C$.  By using $g_{i}$ for all $1\leq i\leq n$, we constructed a continuous compact operator 
$$
\mathcal{K} =\mathcal{K} (F_{n}):  \mathcal{C} (\C)\to \mathcal{C} (\C)
$$ 
in \S 3 of the paper [10]. Similarly, by using $g_{i}\circ \gamma$ for all $1\leq i\leq n$, we have a continuous 
compact operator 
$$
\mathcal{K}_{\gamma} =\mathcal{K} (F_{n}\circ \gamma): \mathcal{C} (\C)\to \mathcal{C} (\C). 
$$

The main point in \S3 of the paper [10] is that we can find the unique fixed point $g_{n+1}$ for $\xi_{n+1}+\mathcal{K}$ and the unique fixed point $g_{n+1, \gamma}$ for $\xi_{n+1}+\mathcal{K}_{\gamma}$. That is,
\begin{equation}~\label{fp1}
g_{n+1}(z) = \xi_{n+1}+ \mathcal{K} g_{n+1}(z)
\end{equation}
and
\begin{equation}~\label{fp2}
g_{n+1, \gamma} (z)= \xi_{n+1}+ \mathcal{K}_{\gamma} g_{n+1,\gamma} (z).
\end{equation}
We also have 
\begin{equation}~\label{fp3}
g_{n+1} \circ \gamma (z)= \xi_{n+1}+ \mathcal{K}_{\gamma} (g_{n+1}\circ \gamma (z)).
\end{equation}

From $g_{n+1}$ (as well as $g_{n+1, \gamma}$), which is holomorphic outside $\Delta$ and continuous in $\C$,  we get a holomorphic map $h_{i}$ (as well as $h_{n+1, \gamma}$) in $\Delta$, for $i=1, \cdots, n$. Then we form two holomorphic maps 
\begin{equation}~\label{nag1}
F_{n+1} = (h_{1} , \cdots , h_{n}, h_{n+1}): \Delta \to Y_{n+1}
\end{equation}
and 
\begin{equation}~\label{nag2}
F_{n+1,\gamma} = (h_{1} \circ \gamma, \cdots, h_{n}\circ \gamma, h_{n+1,\gamma}): \Delta \to Y_{n+1}.
\end{equation}
Since $\Delta$ is simply connected and since $p_{n+1}: T(E_{n+1})\to Y_{n+1}$ is the universal cover, we can lift $F_{n+1}$ and $F_{n+1, \gamma}$ to two two holomorphic maps
$$
f_{\Delta, n+1}: \Delta \to T(E_{n+1})\quad \hbox{and}\quad f_{\Delta, n+1, \gamma}: \Delta \to T(E_{n+1})
$$
such that $p_{n+1}\circ f_{\Delta, n+1}=F_{n+1}$ and 
and  $p_{n+1}\circ f_{\Delta, n+1, \gamma}=F_{n+1, \gamma}$.

Under the assumption of the trivial monodromy, we know that $f_{\Delta, n}=f_{\Delta, n}\circ \gamma$ (see the proof of Theorem A). 
That is, $h_{i}=h_{i}\circ \gamma$ and $g_{i}=g_{i}\circ \gamma$ for all $1\leq i\leq n$.  Thus $\mathcal{K}_{\gamma}=\mathcal{K}$.
Since the fixed point is unique, we get 
$$
g_{n+1}\circ \gamma=g_{n+1, \gamma}=g_{n+1}.
$$ 
This implies that $h_{n+1}\circ \gamma=h_{n+1}$ and 
$F_{n+1}\circ \gamma= F_{n+1}$ and $f_{\Delta, n+1}\circ \gamma=f_{\Delta, n+1}$. 
Since this holds for all $\gamma\in\Gamma$, the map $f_{\Delta, n+1}$ defines a holomorphic map 
$f_{n+1}: X\to T(E_{n+1})$ such that $p\circ f_{n+1} =f_{n}$. Therefore, $f_{n+1}$ is a lift of $f_{n}$.  This completes the proof.  
\end{proof}

\begin{proof}[Proof of Corollary~\ref{exttrivial}]
This follows at once from Theorem~\ref{thmC}, Proposition~\ref{le}, and Theorem~\ref{thmB}.
\end{proof}

\section{Proof of Theorem~\ref{thmD}}

Let $E$ be an arbitrary closed set in $\RS$ such that $0$, $1$, and $\infty$ are in $E$, and let $X$ be a hyperbolic Riemann surface with a basepoint $x_0$. 

\medskip
We will need the following result; see Theorem~\ref{thmC} in~\cite{BJMS}.

\medskip
\begin{theorem}~\label{preext}
Let $\phi: X\times E\to \RS$ be a holomorphic motion. 
Suppose the restriction of $\phi$ to $X\times E'$ extends to 
a holomorphic motion $\widetilde{\phi}: X\times \RS \to \RS$, 
whenever $\{0, 1, \infty\}\subset E'\subset E$ and $E'$ is finite. 
Then $\phi$ can be extended to a holomorphic motion of $\RS$ over $X$.
\end{theorem}

%See~\cite[pages 62-63]{BJMS} for a proof of this theorem.

\begin{proof}[Proof of Theorem~\ref{thmD}]  
Let $\phi: X \times E \to \RS$ be a {\hm} motion with the property that $\phi$ restricted to $X \times E'$ has trivial monodromy, 
where $E'$ is any finite subset of $E$, containing $0$, $1$, and $\infty$.

Fix some $E'\subset E$ such that $E'$ contains the points $0,1, \infty$ and $\phi$ restricted to $X\times E'$ has trivial mondromy. 
Let $E_{0}=E'$. Consider $E_{1} =E_{0} \cup \{\zeta_{1}\}$ for $\zeta_{1}\not\in E_{0}$.
Inductively, consider $E_{n+1}=E_{n}\cup \{\zeta_{n+1}\}$ for $\zeta_{n+1}\not\in E_{n}$ for all $n\geq 0$. Then we eventually get a countable set $E_{\infty}=\cup_{n=0}^{\infty} E_{n}$ in $\RS$. We can assume that $E_{\infty}$ is dense in $\RS$.  

Define $\phi_{0}: =\phi$ restricted to $X\times E_{0}$. Using Theorem~\ref{thmC} and Corollary~\ref{exttrivial} inductively, we know that the {\hm} motion $\phi_{n}: X\times E_{n}\to \RS$ can be extended to a {\hm} motion $\phi_{n+1}: X\times E_{n+1}\to \RS$ with trivial monodromy for all $n\geq 0$. 
Thus we can extend $\phi_{0}$ to a {\hm} $\phi_{\infty}: X\times E_{\infty}\to\RS$. Since $E_{\infty}$ is dense in $\RS$, it can be further extended to a {\hm} motion $\widetilde{\phi}: X\times \RS\to\RS$. The conclusion now follows from Theorem~\ref{preext}.
\end{proof}

\medskip
\begin{remark}  Let $G$ be a group of M\"obius transformations, such that the closed
set $E$ is $G$-invariant. Let $\phi : X \times E \to \RS$ be a $G$-equivariant holomorphic motion;
see Definition~\ref{ge}. If $\phi$ has the property that $\phi$ restricted to $X\times E'$ has trivial
monodromy, where $E'$ is any subset of E, containing the points $0, 1$, and $\infty$, then $\phi$
can be extended to a holomorphic motion $\widetilde{\phi} : X\times \RS\to \RS$ which is also $G$-equivariant.
The proof is very similar to the proof of Theorem 1 in~\cite{EKK} or of Theorem B in~\cite{BJM}.
\end{remark}

\bibliographystyle{amsalpha}

\end{document}